\documentclass{elsarticle}

\usepackage{lineno,hyperref}
\modulolinenumbers[5]

\journal{Statistics \& Probability Letters}
\makeatletter
\def\ps@pprintTitle{%
  \let\@oddhead\@empty
  \let\@evenhead\@empty
  \def\@oddfoot{\reset@font\hfil\thepage\hfil}
  \let\@evenfoot\@oddfoot
}
\makeatother

\biboptions{authoryear}

\usepackage{amsmath,amssymb,amsthm,mathtools}
\allowdisplaybreaks

\usepackage[shortlabels]{enumitem}

\usepackage{multirow,array,booktabs}
    \newcommand{\otoprule}{\midrule[\heavyrulewidth]}
    \newcolumntype{Q}{>{$}l<{$}}
\newcommand*{\abs}[1]{\left\lvert#1\right\rvert}
\newcommand*{\set}[1]{\left\{#1\right\}}

\newcommand{\E}{\mathbf{E}}
\newcommand{\R}{\mathbb{R}}

\newcommand{\Normal}{\mathcal N}
\newcommand{\tod}{\xrightarrow{d}}
\newcommand{\vect}[1]{\mathbf{#1}}
\newcommand{\var}{\mathbf{Var}}

\newtheorem{thm}{Theorem}[section]
\newtheorem{lem}[thm]{Lemma}
\newtheorem{prop}[thm]{Proposition}
\newtheorem{cor}[thm]{Corollary}
\theoremstyle{definition}

\newtheorem*{asm}{Assumption}
\theoremstyle{remark}
\newtheorem{remark}{Remark}

\begin{document}

\begin{frontmatter}

\title{Parameter estimation in CKLS model by continuous observations}

\author[a]{Yuliya Mishura}
\ead{myus@univ.kiev.ua}

\author[a]{Kostiantyn Ralchenko\corref{mycorrespondingauthor}}
\cortext[mycorrespondingauthor]{Corresponding author}
\ead{k.ralchenko@gmail.com}

\author[a]{Olena Dehtiar}
\ead{alenka.degtyar@gmail.com}

\address[a]{Taras Shevchenko National University of Kyiv, Department of Probability Theory, Statistics and Actuarial Mathematics, Volodymyrska 64/13, 01601 Kyiv, Ukraine}

\begin{abstract}
We consider a stochastic differential equation of the form
$dr_t = (a - b r_t) dt + \sigma r_t^\beta dW_t$,
where $a$, $b$ and $\sigma$ are positive constants,  $\beta\in(\frac12,1)$.
We study the estimation of an unknown drift parameter $(a,b)$ by continuous
observations of a sample path $\{r_t, t \in [0,T]\}$.
We prove the strong consistency and asymptotic normality of the maximum likelihood estimator.
We propose another strongly consistent estimator, which generalizes an estimator proposed in \cite{DMR} for $\beta=\frac12$.
The identification of the diffusion parameters $\sigma$ and $\beta$ is discussed as well.
\end{abstract}

\begin{keyword}
CKLS model \sep continuous observations \sep parameter estimation \sep strong consistency \sep asymptotic normality
\MSC[2010]   60H10 \sep 62F10 \sep	62F12 \sep91G70
\end{keyword}

\end{frontmatter}


\section{Introduction}

We consider a stochastic differential equation of the form
\begin{equation}\label{equ:1}
dr_t = (a - b r_t) dt + \sigma r_t^\beta dW_t,
\quad r_t\big|_{t=0} = r_0>0,
\end{equation}
where $W=\{W_t, t \ge 0\}$ is a Wiener process.
This equation is known as
Chan--Karolyi--Longstaff--Sanders (CKLS) model,
which was introduced in \cite{CKLS} for interest rate modeling.
The particular case $\beta=\frac12$ corresponds to the famous Cox--Ingersoll--Ross (CIR) process.
The model \eqref{equ:1} with $\beta = 2/3$ is used as a model for stochastic volatility, see \cite{CarrSun,Lewis}.
The fractional generalization of this model was investigated in \cite{Kubilius21}.

For the case of CIR process, the problem of the drift parameters estimation was extensively studied in the literature, see e.\,g.\  \cite{BenAlaya2012,BenAlaya2013,Rosi2010,DMR,Overbeck1997} and references cited therein.
However in the case $\beta\ne1/2$ we can mention only the work by \cite{Dokuchaev17}, who studied the estimation of the diffusion term parameters $\beta$ and $\sigma$.

In the present paper we concentrate on the case $\beta\in(\frac12,1)$.
First, we extend the results of
\cite{BenAlaya2012,BenAlaya2013} concerning the maximum likelihood estimation of $(a,b)$. Then we generalize an alternative approach to drift parameters estimation developed in the recent paper by \cite{DMR} for the CIR process.

The paper is organized as follows.
In Section~\ref{sec:prel} we present some basic properties of the CKLS model.
Section~\ref{sec:mle} is devoted to the maximum likelihood estimation of drift parameters. An alternative strongly consistent estimators are studied in
Section~\ref{sec:alt}.
In Section~\ref{sec:diff} we discuss the identification of the parameters $\beta$ and $\sigma$.
Some numerical results are presented in Section~\ref{sec:simul}.

\section{Preliminaries}
\label{sec:prel}
In what follows we assume that the following condition is satisfied.
\begin{asm}
$a>0$, $b>0$, $\sigma>0$, $\frac{1}{2}< \beta < 1$, $r_0>0$.
\end{asm}

Our main goal is to estimate parameters $a$ and $b$ of the model \eqref{equ:1} by continuous observations of a trajectory of $r$ on the interval $[0, T]$.
We assume that the parameters $\beta$ and $\sigma$ are known.
This assumption is natural, because $\beta$ and $\sigma$ can be
obtained explicitly from the observations, see Section~\ref{sec:diff} for details.

The next proposition summarizes the properties of the solution of a stochastic differential equation \eqref{equ:1}.
\begin{prop}\label{prop1}
\begin{enumerate}
\item \label{st1}
The equation \eqref{equ:1} has a unique strong solution $r=\{r_t,t\ge0\}$.
\item \label{st2}
The process $r$ is a.\,s.\ strictly positive.
\item \label{st3}
The process $r$ is an ergodic diffusion with the following stationary density:
\begin{equation*}
p_{\infty}(x) =  G \cdot x^{-2\beta} \exp\set{\frac{2}{\sigma^{2}}\left(\frac{a \cdot x^{1-2\beta}}{1-2\beta} -
\frac{b \cdot x^{2-2\beta}}{2-2\beta} \right)}, \quad x>0.
\end{equation*}
where
\begin{equation}\label{equ:G}
G =  \left(\int_{0}^\infty y^{-2\beta} \exp{\frac{2}{\sigma^{2}}\left(\frac{a \cdot y^{1-2\beta}}{1-2\beta} - \frac{b \cdot y^{2-2\beta}}{2-2\beta} \right)} dy \right)^{-1}.
\end{equation}
\item \label{st4}
The integral $\int_{0}^{\infty}x^{\mu}p_{\infty}(x)dx$ is finite for all $\mu\in\R$.
\end{enumerate}
\end{prop}

\begin{proof}
The statements \ref{st1}--\ref{st3} can be found in \cite{AndPit07}.
They are obtained by straightforward application of the general theory of diffusion processes.
Indeed, the existence and uniqueness of a solution follow from the Yamada--Watanabe theorem \cite[Prop.\ 2.13, p. 291]{KarShr91}, the positivity is deduced from Feller's test for explosions \cite[Thm. 5.29, p. 348]{KarShr91}, the third statement is an application of the ergodic theory for homogeneous diffusions \cite[Ch. 1, \S 3]{Skorohod1987}.

In order to prove the statement \ref{st4}, note that the integrand tends to zero as $x\to0$.
Further, taking arbitrary $\lambda>1$, we get as $x\to\infty$
\[
\frac{x^{\mu}p_{\infty}(x)}{x^{-\lambda}}
= G x^{\mu+\lambda-2\beta} \exp\set{\frac{2}{\sigma^{2}}\left(\frac{a \cdot x^{1-2\beta}}{1-2\beta} -
        \frac{b \cdot x^{2-2\beta}}{2-2\beta} \right)}
        \to0.
\]
Therefore, $\int_{\epsilon}^{\infty}x^{\mu}p_{\infty}(x)dx<\infty$, $\epsilon>0$, since
$\int_\epsilon^\infty x^{-\lambda}\,dx<\infty$.
\end{proof}

%

\begin{cor}\label{cor:conv}
The ergodic theorem implies that for all $\mu\ne0$
\begin{multline*}
		 \frac{1}{T}\int_{0}^{T}r_t^{\mu} dt \to \int_{0}^{\infty}x^{\mu}p_{\infty}(x)dx 
		 \\
		 = G\int_{0}^{\infty}x^{\mu-2\beta} \exp\set{\frac{2}{\sigma^{2}}\left(\frac{a \cdot x^{1-2\beta}}{1-2\beta} -
        \frac{b \cdot x^{2-2\beta}}{2-2\beta} \right)}dx
\quad\text{a.\,s.\ as }T\to\infty.
\end{multline*}
\end{cor}

\section{Maximum likelihood estimation of drift parameters}
\label{sec:mle}

In this section we discuss the construction of the maximum likelihood estimators of the drift parameters and their asymptotic behavior.
We consider three cases: estimation of the couple $(a,b)$, estimation of $a$ for known $b$, estimation of $b$ for known $a$.

\subsection{Estimation of the vector parameter $(a,b)$}
Let us start with the construction of the maximum likelihood estimator of the couple $(a, b)$.
	\begin{lem}
The maximum likelihood estimator for the couple $(a,b)$ constructed by the continuous observations of $r$ over the interval $[0,T]$ has the form
		\begin{gather}
		\label{equ:est-a}
	     \hat a_T  =\frac{\int_{0}^{T}\frac{dr_t}{r_t^{2\beta}} \cdot \int_{0}^{T}{r_t^{2-2\beta}dt-\int_{0}^{T}r_t^{1-2\beta}dr_t \cdot \int_{0}^{T} r_t^{1-2\beta}dt} }{\int_{0}^{T}\frac{dt}{r_t^{2\beta}}\cdot \int_{0}^{T}r_t^{2-2\beta}dt-  \left (\int_{0}^{T} r_t^{1-2\beta}dt\right)^2};
	\\
	\label{equ:est-b}
	      \hat b_T =\frac{\int_{0}^{T}\frac{dr_t}{r_t^{2\beta}} \cdot \int_{0}^{T} r_t^{1-2\beta}dt - \int_{0}^{T}r_t^{1-2\beta}dr_t \cdot \int_{0}^{T}\frac{dt}{r_t^{2\beta}} }{\int_{0}^{T}\frac{dt}{r_t^{2\beta}}\cdot \int_{0}^{T}r_t^{2-2\beta}dt-  \left (\int_{0}^{T} r_t^{1-2\beta}dt\right )^2}.
	\end{gather}
	\end{lem}
\begin{proof} Dividing \eqref{equ:1} by $r_t^\beta$ and integrating \eqref{equ:1} over the time interval $[0,s]$ we get  the equality:
	\begin{equation*}
         \int_{0}^{s}\frac{dr_t}{r_t^\beta} = \int_{0}^{s}\left(\frac{a}{r_t^\beta}-b\cdot{{r_t}^{1-\beta}}\right)dt +  \sigma W_s.
    \end{equation*}

	In order to construct likelihood function for the estimation of couple $(a,b)$ of parameters, we
use the  Girsanov theorem for the Wiener process with the drift that equals $$\frac{1}{\sigma}\int_{0}^{s} \left(\frac{a}{r_t^\beta}-b\cdot{{r_t}^{1-\beta}}\right)dt.$$
Then the likelihood function that corresponds  to the likelihood  $dQ_{(0,0)}/dQ_{(a,b)}$, where $Q_{(0,0)}$ is a probability measure responsible for zero values of parameters, and $Q_{(a,b)}$ is responsible for the couple  $(a,b)$, gets the following form:
	\begin{align*}
	     \mathcal{L} & = \exp \left\{ - \int_{0}^{T} \frac{a - b r_t}{\sigma r_t^\beta}dW_t - \frac{1}{2} \int_{0}^{T} \frac{(a - b r_t)^2}{(\sigma {r_t^\beta})^2}dt \right\}\\
	     &= \exp \left\{- \int_{0}^{T} \frac{a - b r_t}{\sigma^2 r_t^{2\beta}}(dr_t - (a - b r_t) dt)
	     - \frac{1}{2} \int_{0}^{T} \frac{(a - b r_t)^2}{(\sigma {r_t^\beta})^2}dt\right\} \\
	     &= \exp \left\{- \int_{0}^{T} \frac{a - b r_t}{\sigma^2 r_t^{2\beta}}dr_t + \frac{1}{2} \int_{0}^{T} \frac{(a - b r_t)^2}{(\sigma {r_t^\beta})^2}dt\right\}.
	\end{align*}
	Since we are interested in maximization of the function that corresponds to $dQ_{(a,b)}/dQ_{(0,0)}$, our likelihood function has the form
\begin{align*}
	     \tilde{\mathcal{L}} & =
	      \exp \left\{\int_{0}^{T} \frac{a - b r_t}{\sigma^2 r_t^{2\beta}}dr_t - \frac{1}{2} \int_{0}^{T} \frac{(a - b r_t)^2}{(\sigma {r_t^\beta})^2}dt\right\}.
	\end{align*}
    As usual, it is computationally convenient to work with the log-likelihood function:
	\begin{equation}\label{equ:2}
	     \log \tilde{\mathcal{L}} =   \int_{0}^{T} \frac{a - b r_t}{\sigma^2 r_t^{2\beta}}dr_t - \frac{1}{2} \int_{0}^{T} \frac{(a - b r_t)^2}{(\sigma {r_t^\beta})^2}dt.
	\end{equation}
	
	In order to determine the maximum value of the function \eqref{equ:2},   we calculate  its first derivatives with respect to $a$ and $b$:
	\begin{align}\label{equ:3}
	     (\log\tilde{\mathcal{L}})_a ^\prime &=   \int_{0}^{T} \left(\frac{a - b r_t}{\sigma^2 r_t^{2\beta}}\right)_a ^\prime dr_t - \frac{1}{2} \int_{0}^{T} \left(\frac{(a - b r_t)^2}{(\sigma {r_t^\beta})^2}\right)_a ^\prime dt
	     \notag\\
	     &=  \frac{1}{\sigma^2}\int_{0}^{T}\frac{dr_t}{r_t^{2\beta}} - \frac{a}{\sigma^2}\int_{0}^{T}\frac{dt}{r_t^{2\beta}}
	      + \frac{b}{\sigma^2}\int_{0}^{T} r_t^{1-2\beta}dt;
	\\
	\label{equ:4}
(\log \tilde{\mathcal{L}})_b ^\prime &=  \int_{0}^{T} \left(\frac{a - b r_t}{\sigma^2 r_t^{2\beta}}\right)_b ^\prime dr_t - \frac{1}{2} \int_{0}^{T} \left(\frac{(a - b r_t)^2}{(\sigma {r_t^\beta})^2}\right)_b ^\prime dt
\notag\\
&= -\frac{1}{\sigma^2}\int_{0}^{T}r_t^{1-2\beta}dr_t+ \frac{a}{\sigma^2}\int_{0}^{T} r_t^{1-2\beta}dt
- \frac{b}{\sigma^2}\int_{0}^{T}r_t^{2-2\beta} dt.
	\end{align}
	
	Now we equate  the derivatives \eqref{equ:3} and  \eqref{equ:4}   to zero and solve the resulting system of equations:
\[	
	\begin{cases}
		 (\log\tilde{\mathcal{L}})_a ^\prime=0,\\
		 (\log\tilde{\mathcal{L}})_b ^\prime = 0.
	\end{cases}
	\Longleftrightarrow
	\hspace{13pt}
	\begin{cases}
	    - a\int_{0}^{T}\frac{dt}{r_t^{2\beta}} + b\int_{0}^{T} r_t^{1-2\beta}dt= -\int_{0}^{T}\frac{dr_t}{r_t^{2\beta}} ,\\
		 a\int_{0}^{T} r_t^{1-2\beta}dt - b\int_{0}^{T}r_t^{2-2\beta} dt=\int_{0}^{T}r_t^{1-2\beta}dr_t,
	\end{cases}
\]  whence \eqref{equ:est-a}  and \eqref{equ:est-b}  follow.
        In order to show that function \eqref{equ:2} attains its maximum value at the point $(\hat a_T, \hat b_T)$,  we calculate its second derivatives:
        \begin{equation}\label{equmin:1}
	        \left (\log \tilde{\mathcal{L}}\right )_{aa} ^{\prime\prime} = - \frac{1}{\sigma^2}\int_{0}^{T}\frac{dt}{r_t^{2\beta}},\quad
	        \left (\log \tilde{\mathcal{L}}\right )_{bb} ^{\prime\prime} = - \frac{1}{\sigma^2}\int_{0}^{T}r_t^{2-2\beta} dt,\quad
	        \left (\log \tilde{\mathcal{L}}\right )_{ba} ^{\prime\prime} = \frac{1}{\sigma^2}\int_{0}^{T} r_t^{1-2\beta}dt.
	   \end{equation}

        Using the Cauchy--Schwarz inequality, we immediately get that
%
	    \begin{equation}\label{equmin:2}
	         (\log \tilde{\mathcal{L}})_{bb} ^{\prime\prime}\cdot (\log \tilde{\mathcal{L}})_{aa} ^{\prime\prime}-((\log \tilde{\mathcal{L}})_{ba} ^{\prime\prime})^2 = \frac{1}{\sigma^4}\left(\int_{0}^{T}\frac{dt}{r_t^{2\beta}}\cdot\int_{0}^{T}r_t^{2-2\beta} dt - \left(\int_{0}^{T} r_t^{1-2\beta}dt\right)^2\right) \ge 0.
	    \end{equation}
Relations \eqref{equmin:1} and \eqref{equmin:2} mean that we have maximum at the point $(\hat a_T, \hat b_T)$, and the proof follows.
\end{proof}

\begin{thm}\label{th:consist}
The estimator $(\hat a_T,\hat b_T)$ is strongly consistent.
\end{thm}
\begin{proof}
Using \eqref{equ:1}, one can deduce from \eqref{equ:est-a} and \eqref{equ:est-b} the following expressions for the errors:
\begin{align}
\label{equ:est-a-2}
\hat a_T - a  &= \sigma\, \frac{\int_{0}^{T}r_t^{-\beta}dW_t \cdot \int_{0}^{T}{r_t^{2-2\beta}dt-\int_{0}^{T}r_t^{1-\beta}dW_t \cdot \int_{0}^{T} r_t^{1-2\beta}dt} }{\int_{0}^{T} r_t^{-2\beta} dt\cdot \int_{0}^{T}r_t^{2-2\beta}dt - \left (\int_{0}^{T} r_t^{1-2\beta}dt\right )^2}
\\
\label{equ:est-b-2}
\hat b_T - b  &= \sigma\,\frac{\int_{0}^{T}r_t^{-\beta}dW_t \cdot \int_{0}^{T}{r_t^{1-2\beta}dt-\int_{0}^{T}r_t^{1-\beta}dW_t \cdot \int_{0}^{T} r_t^{-2\beta}dt} }{\int_{0}^{T} r_t^{-2\beta} dt\cdot \int_{0}^{T}r_t^{2-2\beta}dt - \left (\int_{0}^{T} r_t^{1-2\beta}dt\right )^2}
\end{align}

We prove the strong consistency of $\hat a_T$, and the strong consistency of $\hat b_T$ is established similarly.
Write
\begin{align}
\label{equ:est-a-3}
\hat a_T - a  &= \sigma\, \frac{\frac{1}{T^2}\int_{0}^{T}r_t^{-\beta}dW_t \cdot \int_{0}^{T}{r_t^{2-2\beta}dt-\frac{1}{T^2}\int_{0}^{T}r_t^{1-\beta}dW_t \cdot \int_{0}^{T} r_t^{1-2\beta}dt} }{\frac{1}{T}\int_{0}^{T} r_t^{-2\beta} dt\cdot \frac{1}{T}\int_{0}^{T}r_t^{2-2\beta}dt - \left (\frac{1}{T}\int_{0}^{T} r_t^{1-2\beta}dt\right )^2}
\end{align}

According to Corollary \ref{cor:conv}, the denominator in \eqref{equ:est-a-3} converges to
\begin{align*}
D_\infty = \int_{0}^{T}x^{-2\beta}p_{\infty}(x)dx \cdot \int_{0}^{T}x^{2-2\beta}p_{\infty}(x)dx -  \left (\int_{0}^{T} x^{1-2\beta}p_{\infty}(x)dx \right )^2
\end{align*}
a.\,s.\ as $T\rightarrow\infty$.
By the Cauchy--Schwarz inequality, we have that $D_\infty \ge 0$, where the equality may hold only if the functions $x^{-\beta}$ and $x^{1-\beta}$ are linearly dependent (i.\,e.\ $x^{-\beta} = Cx^{1-\beta}$ a.\,e.\ for some constant $C$), which is impossible. Thus $D_\infty > 0$.

Now it remains to prove that the numerator in \eqref{equ:est-a-3} tends to zero a.\,s.\ as  $T\rightarrow\infty$. The first term in the numerator can be rewritten as follows:
\begin{align*}
\frac{1}{T^2}\int_{0}^{T}r_t^{-\beta}dW_t \cdot \int_{0}^{T}r_t^{2-2\beta}dt = \frac{\int_{0}^{T}r_t^{-\beta}dW_t}{\int_{0}^{T}r_t^{-2\beta}dt} \cdot \frac{1}{T} \int_{0}^{T} r_t^{-2\beta}dt \cdot \frac{1}{T} \int_{0}^{T} r_t^{2-2\beta}dt.
\end{align*}
By Corollary \ref{cor:conv},
\begin{gather}
		 \frac{1}{T}\int_{0}^{T}r_t^{-2\beta} dt \to \int_{0}^{\infty}x^{-2\beta}p_{\infty}(x)dx,
		 \label{equ:conv1}
		\\
		\frac{1}{T}\int_{0}^{T}r_t^{2-2\beta} dt \to \int_{0}^{\infty}x^{2-2\beta}p_{\infty}(x)dx
		\label{equ:conv2}
\end{gather}
a.\,s., as $T\rightarrow\infty$.
Consequently, $\int_{0}^{T}r_t^{-2\beta} dt \to \infty$, a.\,s., as $T\rightarrow\infty$, but $\int_{0}^{T}r_t^{-2\beta} dt$ is a square characteristic of the locally square integrable martingale $\int_{0}^{T}r_t^{-\beta}dW_t$. According to the strong law of large numbers for locally square integrable martingales \cite{LSh2012},
\[
\frac{\int_{0}^{T}r_t^{-\beta}dW_t}{\int_{0}^{T}r_t^{-2\beta}dt} \to 0 \quad\text{a.\,s., as }T\rightarrow\infty.
\]
Hence, the first term in the numerator converges to zero a.\,s., as $T\rightarrow\infty$.
The second term can be considered similarly.
	\end{proof}

\begin{thm}\label{th:as.norm}
The estimator $(\hat a_T,\hat b_T)$ is asymptotically normal:
\begin{align*}
\sqrt{T}
\begin{pmatrix}
\hat a_T - a \\
\hat b_T - b
\end{pmatrix}
\tod \Normal\left (\vect0,\sigma^2 \Sigma^{-1}\right ), \quad T \to \infty,
\end{align*}
where
\begin{align*}
\Sigma =
\begin{pmatrix}
\int_{0}^{\infty}x^{-2\beta}p_{\infty}(x)dx & -\int_{0}^{\infty}x^{1-2\beta}p_{\infty}(x)dx \\[5pt]
-\int_{0}^{\infty}x^{1-2\beta}p_{\infty}(x)dx & \int_{0}^{\infty}x^{2-2\beta}p_{\infty}(x)dx
\end{pmatrix}.
\end{align*}
\end{thm}
\begin{proof}
Introduce the following two-dimensional martingale
\begin{align*}
M_T =
\begin{pmatrix}
\int_{0}^{T}r_t^{-\beta}dW_t \\[5pt]
-\int_{0}^{T}r_t^{1-\beta}dW_t
\end{pmatrix}, \quad T \ge 0,
\end{align*}
with quadratic variation matrix
\begin{align*}
\langle M\rangle_T =
\begin{pmatrix}
\int_{0}^{T}r_t^{-2\beta}dt & -\int_{0}^{T}r_t^{1-2\beta}dt\\[5pt]
-\int_{0}^{T}r_t^{1-2\beta}dt & \int_{0}^{T}r_t^{2-2\beta}dt
\end{pmatrix}.
\end{align*}
It follows from \eqref{equ:est-a-2}--\eqref{equ:est-b-2} that
\begin{align*}
\begin{pmatrix}
\hat a_T - a \\
\hat b_T - b
\end{pmatrix} = \sigma \langle M\rangle_T^{-1} M_T.
\end{align*}
According to Corollary \ref{cor:conv},
\begin{align*}
\frac{\langle M\rangle_T}{T} \to \Sigma \quad\text{a.\,s., as }T\rightarrow\infty.
\end{align*}
Then by the central limit theorem for multidimensional martingales \cite[Thm.12.6]{Heyde}, we have the convergence
\begin{align*}
\sqrt{T} \langle M\rangle_T^{-1} M_T \tod \Normal(\vect0,\Sigma^{-1}),
\quad\text{as } T\to\infty,
\end{align*}
whence the result follows.
\end{proof}

\subsection{Estimation of $b$ when $a$ is known}

The case of known $a$ is considered similarly to the case of two unknown parameters, so we omit some details.

\begin{thm}\label{thm:b-known}
Let $a$ be known. The maximum likelihood estimator for $b$ is
\[
\check{b}_T=
\frac{a\int_0^T r_t^{1-2\beta}\,dt - \int_0^T r_t^{1-2\beta}\,dr_t}{\int_0^T r_t^{2-2\beta}\,dt}
\]
It is strongly consistent and asymptotically normal:
\[
\sqrt{T}\left(\check{b}_T - b\right) \tod \Normal\left(0,
\frac{\sigma^2}{\int_0^\infty x^{2-2\beta} p_\infty(x)\,dx} \right).
\]
\end{thm}

\begin{proof}
Similarly to \eqref{equ:est-a-2}--\eqref{equ:est-b-2}
we can represent  the difference $\check{b}_T - b$ in the following form:
\[
\check{b}_T - b = -\sigma\, \frac{\int_0^T r_t^{1-\beta}\,dW_t}{\int_0^T r_t^{2-2\beta}\,dt},
\]
where the numerator is a martingale, and the denominator is its quadratic variation.
Taking into account \eqref{equ:conv2}, we see that the strong consistency follows from the strong law of large numbers for martingales, and the asymptotic normality is a consequence of the corresponding central limit theorem.
\end{proof}

\subsection{Estimation of $a$ when $b$ is known}

\begin{thm}\label{thm:a-known}
Let $b$ be known. The maximum likelihood estimator for $a$ is
\[
\check{a}_T=
\frac{b\int_0^T r_t^{1-2\beta}\,dt + \int_0^T r_t^{-2\beta}\,dr_t}{\int_0^T r_t^{-2\beta}\,dt}
\]
It is strongly consistent and asymptotically normal:
\[
\sqrt{T}\left(\check{a}_T - a\right) \tod \Normal\left(0,
\frac{\sigma^2}{\int_0^\infty x^{-2\beta} p_\infty(x)\,dx} \right).
\]
\end{thm}

The proof is conducted similarly to that of Theorem~\ref{thm:b-known},
taking into account the representation
\[
\check{a}_T - a = \sigma\, \frac{\int_0^T r_t^{-\beta}\,dW_t}{\int_0^T r_t^{-2\beta}\,dt}.
\]

\begin{remark}
The results of Section~\ref{sec:mle} remain true in the case $\beta=\frac12$ under the additional condition $2a>\sigma^2$.
In this case $p_\infty(x)$ has the form
\[
p_{\infty}(x)=\frac{\beta^\alpha}{\Gamma(\alpha)}x^{\alpha-1}e^{-\beta x}, \quad x>0, \quad \text{with }\alpha=\frac{2a}{\sigma^2},\;\beta=\frac{2b}{\sigma^2}.
\]
The details can be found in \cite{BenAlaya2012,BenAlaya2013}, where the asymptotic distributions of the estimators in the boundary cases $b=0$ and $2a=\sigma^2$ are investigated as well.
\end{remark}

\section{An alternative approach to drift parameters estimation}
In this section we generalize an approach to drift parameters estimation proposed in \cite{DMR} for the CIR process.
We start with two auxiliary lemmas.
\label{sec:alt}
	\begin{lem}\label{conv31}
One has the following convergence
\begin{align}\label{equ:A}
    \frac{1}{T}\int_{0}^{T}r_t dt \to \frac{a}{b} \quad\text{a.\,s., as }T\rightarrow\infty.
\end{align}
	\end{lem}
\begin{proof}
By Corollary \ref{cor:conv},
\begin{align*}
    \frac{1}{T}\int_{0}^{T}r_t dt \to \int_{0}^{\infty}x p_{\infty}(x)dx \quad\text{a.\,s., as }T\rightarrow\infty.
\end{align*}
Integrating by parts, we obtain
\begin{align}\label{equ:B}
		\int_{0}^{\infty}x p_{\infty}(x)dx &= G \int_{0}^{\infty}x^{1-2\beta} \exp\set{\frac{2}{\sigma^2} \left (\frac{ax^{1-2\beta}}{1-2\beta}-\frac{bx^{2-2\beta}}{2-2\beta} \right )}dx {}\nonumber\\
		& =-\frac{G\sigma^2}{2b}\int_{0}^{\infty} \exp\set{ \frac{2}{\sigma^2}\cdot\frac{ax^{1-2\beta}}{1-2\beta}}d\exp\set{-\frac{2}{\sigma^2}\cdot\frac{bx^{2-2\beta}}{2-2\beta}} {}\nonumber\\
		&= -\frac{G\sigma^2}{2b}\cdot\exp\set{\frac{2}{\sigma^2} \left (\frac{ax^{1-2\beta}}{1-2\beta}-\frac{bx^{2-2\beta}}{2-2\beta} \right )}\bigg|_{x=0}^{\infty}
		{}\nonumber\\
		&+ \frac{G\sigma^2}{2b}\int_{0}^{\infty}\exp\set{ -\frac{2}{\sigma^2}\frac{bx^{2-2\beta}}{2-2\beta}\cdot}d\exp\set{\frac{2}{\sigma^2}\cdot\frac{ax^{1-2\beta}}{1-2\beta}}
		{}\nonumber\\
		&= \frac{Ga}{b}\int_{0}^{\infty}x^{-2\beta} \exp\set{\frac{2}{\sigma^2} \left (\frac{ax^{1-2\beta}}{1-2\beta}-\frac{bx^{2-2\beta}}{2-2\beta} \right )}dx = \frac{a}{b},
\end{align}
where the last equality follows from \eqref{equ:G}.
\end{proof}

	\begin{lem}\label{conv32}
One has the following convergence
\begin{align}\label{equ:C}
    \frac{1}{T}\int_{0}^{T}r_t^{3-2\beta} dt -
    \frac{a}{bT}\int_{0}^{T}r_t^{2-2\beta} dt \to \frac{\sigma^{2}(1-\beta)a}{b^2} \quad\text{a.\,s., as }T\rightarrow\infty.
\end{align}
	\end{lem}
\begin{proof}
By Corollary \ref{cor:conv},
\begin{align}\label{equ:D}
    \frac{1}{T}\int_{0}^{T}r_t^{3-2\beta} dt -
    \frac{a}{bT}\int_{0}^{T}r_t^{2-2\beta} dt \to \int_{0}^{\infty}x^{3-2\beta} p_{\infty}(x)dx -
    \frac{a}{b}\int_{0}^{\infty}x^{2-2\beta} p_{\infty}(x)dx
\end{align}
a.\,s., as $T\rightarrow\infty$.
Using integration-by-parts, we get
\begin{align}\label{equ:E}
\MoveEqLeft
		\int_{0}^{\infty}x^{3-2\beta} p_{\infty}(x)dx = G \int_{0}^{\infty}x^{3-4\beta} \exp\set{\frac{2}{\sigma^2} \left (\frac{ax^{1-2\beta}}{1-2\beta}-\frac{bx^{2-2\beta}}{2-2\beta} \right )}dx {}\nonumber\\
		& =-\frac{G\sigma^2}{2b}\int_{0}^{\infty}x^{2-2\beta} \exp\set{ \frac{2}{\sigma^2}\cdot\frac{ax^{1-2\beta}}{1-2\beta}}d\exp\set{-\frac{2}{\sigma^2}\cdot\frac{bx^{2-2\beta}}{2-2\beta}} {}\nonumber\\
		& =\frac{G\sigma^2}{2b}\int_{0}^{\infty}\exp\set{- \frac{2}{\sigma^2}\cdot\frac{bx^{2-2\beta}}{2-2\beta} }d\left (x^{2-2\beta}\exp\set{\frac{2}{\sigma^2}\cdot\frac{ax^{1-2\beta}}{1-2\beta}}\right )
		{}\nonumber\\	
		& =\frac{G\sigma^2}{2b}(2-2\beta)\int_{0}^{\infty}x^{1-2\beta}\exp\set{\frac{2}{\sigma^2} \left (\frac{ax^{1-2\beta}}{1-2\beta}-\frac{bx^{2-2\beta}}{2-2\beta} \right )}dx
		{}\nonumber\\
		&\quad+ \frac{Ga}{b}\int_{0}^{\infty}x^{2-4\beta}\exp\set{\frac{2}{\sigma^2} \left (\frac{ax^{1-2\beta}}{1-2\beta}-\frac{bx^{2-2\beta}}{2-2\beta} \right )}dx
		{}\nonumber\\
		&= \frac{\sigma^{2}a(1-\beta)}{b^2}+\frac{a}{b}\int_{0}^{\infty}x^{2-2\beta} p_{\infty}(x)dx,
\end{align}
where the last equality follows from \eqref{equ:B}.

Combining \eqref{equ:D} and \eqref{equ:E} we conclude the proof.
\end{proof}

The convergences \eqref{equ:A} and \eqref{equ:C} enable us to construct the estimators of the drift parameters by solving the corresponding system of equations. We obtain the following estimators:
\begin{gather*}
	     \tilde{a}_T
	     =\frac{\sigma^{2}(1-\beta)\big(\int_{0}^{T}{r_t}dt\big)^2} {T\int_{0}^{T}r_t^{3-2\beta}dt-  \int_{0}^{T} r_tdt \cdot \int_{0}^{T}{r_t^{2-2\beta}dt}};
	\\
	      \tilde{b}_T =\frac{\sigma^{2}(1-\beta)T\int_{0}^{T} {r_t} dt }{T\int_{0}^{T}r_t^{3-2\beta}dt-  \int_{0}^{T} r_tdt \cdot \int_{0}^{T}{r_t^{2-2\beta}dt}}.
\end{gather*}

\begin{thm}\label{th:consist1}
 $(\tilde{a}_T,\tilde{b}_T)$ is a strongly consistent estimator of the parameter $(a, b)$
\end{thm}
\begin{proof}
Lemmas \ref{conv31} and \ref{conv32} imply that
\begin{align*}
    \frac{1}{T}\int_{0}^{T}r_t^{3-2\beta} dt -
    \frac{1}{T^2}\int_{0}^{T}r_t dt\cdot\int_{0}^{T}r_t^{2-2\beta} dt \to \frac{\sigma^{2}(1-\beta)a}{b^2} \quad\text{a.\,s., as }T\rightarrow\infty.
\end{align*}
Now the proof follows from Lemma \ref{conv31}.
\end{proof}

\begin{remark}
Theorem \ref{th:consist1} remains true if $\beta = \frac{1}{2}$. This case was studied in \cite[Thm.5]{DMR}.
\end{remark}

\section{Estimation of $\beta$ and $\sigma$}
\label{sec:diff}

It turns out that if we observe the whole path $\set{r_t, t \in[0, T]}$, then the parameters $\beta$ and $\sigma$ can be obtained explicitly from the observations by considering realized quadratic variations.
First, we assume that the parameter $\sigma$ is known.
The next proposition gives the value of $\beta$.
\begin{prop}
For any $t\in[0,T]$,
\begin{equation}\label{equ:beta-est1}
\beta = \lim_{h\to 0}\frac{\log\left (\frac{[r]_{t+h}-[r]_{t}}{\sigma^2 h}\right )}{2\log r_{t}}
\quad\text{a.\,s.,}
\end{equation}
where
\[
[r]_t = \lim_{n\to\infty}\sum_{k=1}^{2^n} \left(r_{\frac{kt}{2^n}}- r_{\frac{(k-1)t}{2^n}}\right)^2.
\]
\end{prop}

\begin{proof}
Since
\begin{equation}\label{equ:4-1}
[r]_t = \sigma^2 \int_0^t r_s^{2\beta} \, ds,
\end{equation}
we see that
\begin{equation}\label{equ:4-2}
\frac{[r]_{t+h}-[r]_{t}}{h}
\to \sigma^2 r_{t}^{2\beta},
\quad\text{a.\,s.\ as } h \to 0.
\end{equation}
Therefore
\begin{equation}\label{equ:beta-est1a}
\log\left (\frac{[r]_{t+h}-[r]_{h}}{\sigma^2h}\right )
\to  \beta \log r_{t}^2,
\quad\text{a.\,s.\ as } h \to 0,
\end{equation}
which is equivalent to \eqref{equ:beta-est1}.
\end{proof}

Now let $\sigma$ be unknown. In this case we can consider \eqref{equ:4-2} at two different times $t$ and $s$, this allows us to exclude the unknown parameter $\sigma$ and to obtain the following result.
\begin{prop}
For any $s,t\in[0,T]$, $s\ne t$,
\begin{equation}\label{equ:beta-est2}
\beta = \lim_{h\to 0}\frac{\log\left (\frac{[r]_{t+h}-[r]_{t}}{[r]_{s+h}-[r]_{s}}\right )}{2\log(r_{t}/r_s)}
\quad\text{a.\,s.}
\end{equation}
\end{prop}

\begin{proof}
The convergence \eqref{equ:4-2} yields
\[
\frac{[r]_{t+h}-[r]_{t}}{[r]_{s+h}-[r]_{s}}
\to \left(\frac{r_{t}}{r_{s}}\right)^{2\beta},
\quad\text{a.\,s.\ as } h \to 0,
\]
whence the proof follows.
\end{proof}

Finally, after $\beta$ is known, the parameter $\sigma$ can be evaluated using either \eqref{equ:4-1} or \eqref{equ:4-2}.
\begin{prop}
For any $t\in[0,T]$,
\begin{equation}\label{equ:sigma-est}
\sigma^2 = \frac{[r]_t}{\int_0^t r_s^{2\beta}\,ds}
=\lim_{h\to 0}\frac{[r]_{t+h}-[r]_{t}}{h r_t^{2\beta}}
\quad\text{a.\,s.}
\end{equation}
\end{prop}

\begin{remark}
Let us give some practical recommendations concerning evaluation of $\beta$ and $\sigma$ from data. Note that the right-hand side of \eqref{equ:beta-est1} can be calculated using the values of the process $r$ in the neighborhood of any $t\in[0,T]$. To avoid large estimation errors, the value of $r_t$ should not be close to $1$ for chosen $t$. In order to obtain better numerical results we recommend to take several distinct points $t_1,\dots, t_m$ and use the following quantity instead of \eqref{equ:beta-est1}: 
\[
\hat\beta_1 (h) = \frac{\sum_{i=1}^m \abs{\log\left (\frac{[r]_{t_i+h}-[r]_{t_i}}{\sigma^2 h}\right )}}{2\sum_{i=1}^m\abs{\log r_{t_i}}}.
\] 
The convergence $\hat\beta_1 (h)\to 0$, $h\to0$, follows from \eqref{equ:beta-est1a}.
Compared to an average value of the estimators obtained from \eqref{equ:beta-est1} at points  $\set{t_i}$,
the estimator $\hat\beta_1 (h)$ performs much better in the situation when some values of $r_{t_i}$ may be close to~1.
Similarly, the following estimators may be used instead of \eqref{equ:beta-est2} and \eqref{equ:sigma-est}:
\[
\hat\beta_2 (h) = \frac{\sum_{i=1}^m\abs{\log\left (\frac{[r]_{t_i+h}-[r]_{t_i}}{[r]_{s_i+h}-[r]_{s_i}}\right )}}{2\sum_{i=1}^m\abs{\log (r_{t_i}/r_{s_i})}},
\qquad
\hat\sigma^2(h) = \frac{\sum_{i=1}^m([r]_{t_i+h}-[r]_{t_i})}{h \sum_{i=1}^m r_{t_i}^{2\beta}}.
\]
\end{remark}

\begin{remark}
An alternative method for evaluation of $\beta$ and $\sigma$ was developed by \cite{Dokuchaev17}.
\end{remark}

\section{Simulation}
\label{sec:simul}

For numerical simulation, we take $a=3$, $b=2$, $\sigma=1$, $r_0=0$ in each procedure.
For each value of $\beta\in\set{0.5,0.6,0.7,0.8,0.9}$, we generate 100 sample paths of the solution $r=\set{r_t,t\in[0,T]}$ to the equation \eqref{equ:1} using Euler's approximation.
The sample means and standard deviations of all estimators for the parameters $a$ and $b$ at the times $T=50,100,150,200$ are reported in Table~\ref{tab:1}.
Observe that all estimators converge to the true values of parameters, however, in most cases the maximum likelihood estimators perform better than the alternative ones. The best results are obtained for $\check a$ and $\check b$, the maximum likelihood estimators in the case when one of the parameters in known.

In Table~\ref{tab:2} we investigate the behavior of the estimators for the parameters $\beta$ and $\sigma$. In order to calculate realized quadratic variation we used a partition with the step $2^{-14}$. The small parameter $h$ was set to $2^{-6}$. In order to evaluate $\hat\beta_1$ and $\hat\sigma^2$ we used eight points $t_i = i/8$, $i=1,\dots, 8$. The estimator $\hat\beta_2$ was calculated using the pairs $s_i = i/16$, $t_i = (i+8)/16$, $i=1,\dots,8$.

\setlength{\tabcolsep}{3.8pt}
\begin{table}
\centering
\caption{Estimates of $a$ and $b$\label{tab:1}}
\footnotesize
\begin{tabular}{Q @{\hspace{13pt}} *{6}{Q}@{\hspace{12pt}}*{5}{Q}}\toprule
&& T &&&&&& T &&&\\
\cmidrule(llll){3-6}\cmidrule(l){9-12}
&&   50 & 100 & 150 & 200 && &  50 & 100 & 150 & 200\\\otoprule
\beta = 0.5& \E [\hat a_T]  &  3.104 &  3.023 &  3.023 &   3.030
&& \E [\hat b_T]   &  2.077 &  2.022 &  2.027 &  2.018\\
&   \var[\hat a_T] &   0.182 &  0.059 &  0.047 &  0.049
&&    \var[\hat b_T] &  0.092 &   0.030 &  0.029 &  0.025 \\
\addlinespace
& \E [\check a_T]  &  3.005 &  2.997 &   2.990 &  3.007
&& \E [\check b_T] &  2.008 &  2.007 &  2.011 &  1.998\\
& \var[\check a_T] &  0.024 &  0.016 &  0.007 &  0.008
&&\var[\check b_T] & 0.012 &  0.008 &  0.004 &  0.004\\
\addlinespace
& \E [\tilde a_T]   &  3.197 &  2.993 &  3.043 &  2.998
&&\E [\tilde b_T]  &   2.160 &  1.993 &  2.037 &  2.001\\
& \var[\tilde a_T]  &  0.463 &  0.307 &  0.276 &   0.220
&&    \var[\tilde b_T] & 0.306 &  0.222 &  0.204 &  0.159\\
\midrule

\beta = 0.6&\E [\hat a_T] &   3.152 &  3.053 &  2.999 &  3.043
&& \E [\hat b_T] &  2.101 &  2.035 &  1.992 &  2.036\\
&   \var[\hat a_T] &   0.147 &  0.069 &  0.038 &  0.038
&&    \var[\hat b_T] &  0.071 &  0.038 &   0.020 &  0.023\\
\addlinespace
& \E [\check a_T]   &  3.029 &  3.011 &  3.008 &  3.001
&& \E [\check b_T]  &  1.998 &  1.999 &  1.993 &  2.006\\
& \var[\check a_T] &  0.026 &  0.014 &  0.007 &  0.007
&& \var[\check b_T] & 0.013 &  0.008 &  0.004 &  0.004\\
\addlinespace
& \E [\tilde a_T]  &  3.186 &  2.996 &  3.043 &  2.993
&& \E [\tilde b_T] &  2.164 &  2.005 &   2.030 &  2.002\\
& \var[\tilde a_T] &  0.495 &  0.308 &  0.278 &  0.221
&&    \var[\tilde b_T] &  0.354 &  0.207 &  0.205 &  0.166\\
\midrule

\beta = 0.7& \E [\hat a_T]   &  3.099 &  3.099 &   3.060 &  3.006
&& \E [\hat b_T]  & 2.075 &  2.085 &  2.049 &  2.021\\
&   \var[\hat a_T] &   0.121 &  0.058 &  0.062 &  0.041
&&    \var[\hat b_T] & 0.084 &  0.033 &  0.041 &  0.024\\
\addlinespace
& \E [\check a_T]  &   3.014 &  3.002 &  3.004 &  2.982
&& \E [\check b_T] &  2.003 &  2.015 &  2.006 &  2.017\\
& \var[\check a_T]  &  0.019 &  0.013 &  0.009 &  0.008
&&    \var[\check b_T]&  0.014 &  0.008 &  0.006 &  0.005\\
\addlinespace
& \E [\tilde a_T]  &   3.141 &  3.085 &  3.002 &  3.056
&& \E [\tilde b_T] &  2.114 &  2.061 &   2.000 &  2.043\\
& \var[\tilde a_T]  &  0.488 &  0.354 &  0.277 &  0.241
&& \var[\tilde b_T] &  0.376 &  0.264 &  0.211 &  0.188\\
\midrule

\beta = 0.8& \E [\hat a_T]   &  3.157 &  3.058 &  3.032 &  3.032
&& \E [\hat b_T] &  2.133 &  2.037 &  2.022 &  2.013\\
&   \var[\hat a_T] &  0.123 &   0.070 &  0.044 &  0.036
&&    \var[\hat b_T] &  0.098 &  0.055 &  0.033 &  0.022\\
\addlinespace
& \E [\check a_T]  &  3.016 &  3.019 &  3.009 &  3.017
&& \E [\check b_T]  &  2.013 &  1.992 &  1.997 &   1.99\\
& \var[\check a_T] &  0.029 &  0.011 &  0.008 &  0.007
&&    \var[\check b_T] &  0.022 &  0.008 &  0.006 &  0.004\\
\addlinespace
& \E [\tilde a_T]  &   3.154 &  3.115 &  3.044 &  3.037
&& \E [\tilde b_T]  &  2.108 &  2.098 &   2.040 &  2.034\\
& \var[\tilde a_T]  &  0.496 &  0.348 &  0.254 &  0.246
&&    \var[\tilde b_T] &  0.397 &   0.270 &  0.195 &   0.190\\
\midrule

\beta = 0.9& \E [\hat a_T]  &  3.129 &  3.034 &   3.050 &  3.015
&& \E [\hat b_T] &  2.113 &  2.029 &  2.048 &  2.006\\
&   \var[\hat a_T] &   0.103 &  0.072 &  0.038 &  0.026
&&    \var[\hat b_T] &  0.102 &  0.055 &  0.029 &  0.024\\
\addlinespace
& \E [\check a_T]  &  3.009 &  3.005 &  3.001 &  3.009
&& \E [\check b_T] &  2.012 &  2.002 &  2.008 &  1.994\\
& \var[\check a_T] &  0.024 &  0.015 &  0.009 &  0.006
&&    \var[\check b_T] &  0.016 &  0.011 &  0.007 &  0.005\\
\addlinespace
& \E [\tilde a_T]  &  3.197 &  3.059 &  3.124 &  3.044
&& \E [\tilde b_T] &   2.180 &  2.043 &  2.086 &  2.037\\
& \var[\tilde a_T]  &   0.579 &  0.384 &  0.335 &  0.251
&&    \var[\tilde b_T] &  0.486 &  0.323 &  0.264 &  0.204\\
\bottomrule
\end{tabular}
\end{table}

\setlength{\tabcolsep}{6pt}
\begin{table}
\centering
\caption{Estimates of $\beta$ and $\sigma^2$ \label{tab:2}}
\footnotesize
\begin{tabular}{*{7}{Q}}
\toprule
& \E[\hat\beta_1] & \var[\hat\beta_1] & \E[\hat\beta_2] & \var[\hat\beta_2] &\E[\hat\sigma^2] & \var[\hat\sigma^2] \\\otoprule
\beta = 0.5 & 0.5246 & 0.0067 & 0.5235 & 0.0070 & 1.0016 & 0.0013 \\
\beta = 0.6 & 0.6115 & 0.0062 & 0.6203 & 0.0071 & 1.0046 & 0.0026 \\
\beta = 0.7 & 0.7077 & 0.0085 & 0.7107 & 0.0083 & 1.0052 & 0.0024 \\
\beta = 0.8 & 0.8181 & 0.0119 & 0.8210 & 0.0202 & 1.0114 & 0.0030 \\
\beta = 0.9 & 0.9008 & 0.0095 & 0.9227 & 0.0157 & 1.0100 & 0.0040\\
\bottomrule
\end{tabular}
\end{table}

\section*{Acknowledgments}
This work was supported by the National Research Fund of Ukraine under Grant 2020.02/0026.

\end{document}